\documentclass{amsart}
\usepackage[all]{xypic}

\newtheorem{theorem}{Theorem}[section]
\newtheorem{lemma}[theorem]{Lemma}
\newtheorem{corollary}[theorem]{Corollary}
\newtheorem{proposition}[theorem]{Proposition}

\theoremstyle{definition}
\newtheorem{problem}[theorem]{Problem}

\newcommand\CB{{\mathcal B}} 
\newcommand\CC{{\mathcal C}}
\newcommand\CD{{\mathcal D}}

\newcommand\CS{{\mathcal S}}

\newcommand\BBC{{\mathbb C}}

\newcommand\BBN{{\mathbb N}}

\newcommand\TT{{T}}

\DeclareMathOperator{\ct}{ct}
\DeclareMathOperator{\End}{End} 
\DeclareMathOperator{\GL}{GL}
\DeclareMathOperator{\Hom}{Hom} 
\DeclareMathOperator{\Ind}{Ind}
\DeclareMathOperator{\Lie}{Lie}
\DeclareMathOperator{\maj}{maj} 
\DeclareMathOperator{\spn}{span}

\newcommand\inverse{^{-1}}
\renewcommand\th{{^{\text{th}}}}

\newcommand\op{^{\textrm{op}}}

\newcommand\SYT{{\mathrm SYT}}

\begin{document}


\title[Schur-Weyl duality and the free Lie algebra]{Schur-Weyl
  duality and the free Lie algebra} 

\author{Stephen Doty}
\address{Department of Mathematics and Statistics\\ Loyola University
  Chi\-cago\\ Chicago, IL 60660, USA}
\email{doty@math.luc.edu} 

\thanks{This work was partially supported by grants from the Simons
  Foundation (Grant \#245975 to S.~Doty and \#245399 to J.M.~
  Douglass). J.M.~Douglass would like to acknowledge that some of this
  material is based upon work supported by (while serving at) the National
  Science Foundation.}

\author{J.~Matthew Douglass}
\address{Department of Mathematics\\
University of North Texas\\ 
Denton, TX 76203, USA}
\email{douglass@unt.edu}



\begin{abstract}\noindent
  We prove an analogue of Schur-Weyl duality for the space of homogeneous
  Lie polynomials of degree $r$ in $n$ variables.
\end{abstract}

\maketitle


\section{Introduction}\label{sec:1}

Let $k$ be a commutative ring and $V$ a given $k$-module. Put $E
= \End_k(V)$. The centralizer of a set $X \subseteq E$ of $k$-linear
endomorphisms of $V$ is the set
\begin{equation*}
  Z_{E}(X) = \{ f \in E \mid fx=xf, \text{ for all } x \in X\}. 
\end{equation*}
Suppose further that $V$ has a given $(A,B)$-bimodule structure, where $A$
and $B$ are $k$-algebras. Let $\overline{A}$, $\overline{B}
\subset \End_k(V)$ be the sets of $k$-linear endomorphisms of $V$ induced by
the actions of $A$ and $B$, respectively. Since the actions of $A$ and $B$
commute, we have inclusions
\begin{equation}\label{eq:SWD}
  \overline{A} \subseteq \End_B(V)\quad \text{and} \quad \overline{B}
  \subseteq \End_A(V),
\end{equation}
where $\End_A(V) = Z_E(\overline{A})$ and $\End_B(V) = Z_E(\overline{B})$.
When the inclusions in \eqref{eq:SWD} are equalities then we say that the
triple $(A,V,B)$ satisfies Schur-Weyl duality. This implies that
$\overline A$ and $\overline B$ both have the double centralizer property,
that is,
\begin{equation}\label{eq:DCP}
  Z_E(Z_E( \overline{A})) = \overline{A}\quad \text{and}\quad Z_E(Z_E(
  \overline{B})) = \overline{B}.
\end{equation}
If $\overline A$ has the double centralizer property and $\overline
B=Z_E(\overline A) = \End_A(V)$, then $\overline A=Z_E(\overline
B)= \End_A(V)$ as well, and $(A,V,B)$ satisfies Schur-Weyl
duality. But the two equalities in \eqref{eq:DCP} do not by themselves
imply that $(A,V,B)$ satisfies Schur-Weyl duality.

Assume henceforth that $k$ is a field. An important example of the duality
above is given by the $r\th$ tensor power $V = \TT^r(k^n) = (k^n)^{\otimes
  r}$ of the space $k^n$ of $n$-dimensional column vectors, regarded as an
$(A,B)$-bimodule, where $A = k\GL_n(k)$ and $B=k\Sigma_r$ are respectively
the group algebras of the general linear group $\GL_n(k)$ and symmetric
group $\Sigma_r$, with $\GL_n(k)$ acting diagonally on the left and
$\Sigma_r$ acting on the right by place permutation. To be precise, the
commuting actions are given by
\[
g\cdot (v_1\otimes \dotsm \otimes v_r) = gv_1\otimes \dotsm \otimes
gv_r\quad\text{and}\quad (v_1\otimes \dotsm \otimes v_r) \cdot \sigma =
v_{\sigma(1)}\otimes \dotsm \otimes v_{\sigma(r)}
\]
for all $v_1, \dots, v_r\in k^n$, $g\in \GL_n(k)$, and $\sigma\in \Sigma_r$.
In this setting the assertion that the triple
$(k\GL_n(k), \TT^r(k^n), k\Sigma_r)$ satisfies Schur-Weyl duality is the
classical Schur-Weyl duality between representations of $\GL_n(k)$ and
$\Sigma_r$, which is known to hold whenever $|k|>r$. (See
\cite{bensondoty:schur} for a proof.)

In this note, we investigate the analogue of classical Schur-Weyl duality
when tensor space $\TT^r(k^n)$ is replaced by its intersection $L^r(k^n)$
with the free Lie algebra $L(k^n)$ on $k^n$.

Recall that the free Lie algebra $L(k^n)$ is the Lie subalgebra of the
tensor algebra $\TT(k^n)$ generated by $k^n$, where $\TT(k^n)$ is regarded
as a Lie algebra via the Lie bracket $[a,b] = ab-ba$. Fixing a basis $X$ of
$k^n$, we may identify the tensor algebra $\TT(k^n)$ with the free
associative algebra $k\langle X\rangle$ on $X$. In this point of view,
elements of $k\langle X\rangle$ are regarded as noncommutative polynomials
in ``variables'' $X$, and polynomials in the subspace $L(k^n)$ are known as
Lie polynomials. The grading $\TT(k^n) = \bigoplus_{r \ge 0} \TT^r(k^n)$
induces a corresponding grading
\begin{equation*}
  L(k^n) = \textstyle \bigoplus_{r \ge 0} L^r(k^n), \quad \text{where}\quad
  L^r(k^n) = L(k^n) \cap \TT^r(k^n)
\end{equation*}
on $L(k^n)$. The $r$th graded component $L^r(k^n)$ in the above
decomposition is the space of homogeneous Lie polynomials of degree $r$.

Since $g\cdot [a,b] = [g\cdot a, g\cdot b]$ for any $g \in \GL_n(k)$,
$a,b \in k^n$, it is clear that $L(k^n)$ is invariant under the action of
$\GL_n(k)$, hence is a left $k\GL_n(k)$-module. It follows that the same
holds for $L^r(k^n)$. A natural problem is to describe the centralizer
algebra $\End_{\GL_n(k)}(L^r(k^n))$ as a subquotient of $k\Sigma_r$; ideally
one would like to identify a subalgebra $B_r$ of $k\Sigma_r$ which maps onto
this centralizer. Furthermore, having identified such a subalgebra $B_r$, it
is natural to ask whether the triple $(k\GL_n(k), L^r(k^n), B_r)$ satisfies
Schur-Weyl duality. When it does, we have an analogue of Schur-Weyl duality
for the module $L^r(k^n)$ of homogeneous Lie polynomials of degree $r$.

Our main results explain how to identify the appropriate subalgebra
$B_r$ and establish that, indeed, $(k\GL_n(k), L^r(k^n), B_r)$ satisfies
Schur-Weyl duality, provided only that the characteristic of $k$ is strictly
larger than $r$. (We agree that characteristic zero is infinite and hence
always larger than any $r$.)  Furthermore, under our assumption on the
characteristic, it is well-known that an idempotent $e \in k\Sigma_r$ exists
such that $L^r(k^n) = \TT^r(k^n)e$. Then $B_r$ may be taken to be the
subalgebra $ek\Sigma_r e$ of $k\Sigma_r$. The more general question, of
whether Schur-Weyl duality holds whenever $|k|>r$ and the characteristic of
$k$ does not divide $r$, remains open.

The paper is organized as follows. In \S\ref{sec:2} we state the main
theorem, Theorem~\ref{thm:1}, show that the triple
$(k\GL_n(k), L^{r}(k^n), ek\Sigma_re)$ satisfies Schur-Weyl duality when the
characteristic of $k$ is strictly larger than $r$, and draw some general
conclusions.

The proof of Theorem~\ref{thm:1} is given in \S\ref{sec:3} as an application
of more general results. We consider a triple $(A,V,B)$ that satisfies
Schur-Weyl duality and an idempotent $e\in B$, and we ask when
$(A, Ve, eBe)$ satisfies Schur-Weyl duality. It turns out that
$\overline {eBe}$ is always equal to $\End_A(Ve)$. The equality
$\overline A= \End_{eBe}(Ve)$ seems to be a more delicate question. We show
that if $V$ is a completely reducible $A$-module whose irreducible
constituents are absolutely irreducible, then in fact $(A, Ve, eBe)$
satisfies Schur-Weyl duality.

In the case of classical Schur-Weyl duality, $L^{r}(k^n)$ is a tilting
module and it is tempting to try to use the theory of tilting modules and a
$p$-modular system to derive results in positive characteristic from known
results in characteristic zero. This approach can be used to yield a uniform
proof in all characteristics of some of the known properties of the algebra
$ek\Sigma_re$, but we have been unable to show that
$(k\GL_n(k), L^{r}(k^n), ek\Sigma_re)$ satisfies Schur-Weyl duality in
general.

Finally, in \S\ref{sec:4} we describe the commuting algebra $ek\Sigma_r e$
in the favorable case when the characteristic of $k$ is strictly larger than
$r$, and we show that in general, whether
$(k\GL_n(k), L^{r}(k^n), ek\Sigma_re)$ satisfies Schur-Weyl duality may be
reduced to a statement about permutation representations of $\Sigma_r$
arising from Young subgroups.


\section{Notation and Main Results}\label{sec:2}

In this section we establish our basic notation and formulate the
main results.

Recall that $k$ denotes a field. Throughout the paper, we set $T^{n,r} =
\TT^r(k^n)$ and $L^{n,r} = L^r(k^n)$ for ease of notation, and we denote by
$\Phi$ and $\Psi$ the $k$-algebra homomorphisms
\begin{equation*}
  \xymatrix{k\GL_n(k) \ar[r]^-{\Phi} &\End_k(T^{n,r})& k\Sigma_r
    \ar[l]_-{\Psi}}
\end{equation*}
induced by the commuting actions of $\GL_n(k)$ and $\Sigma_r$ described in
the introduction.  Note that because $\Sigma_r$ acts on the right, the
homomorphism $\Psi$ is given by $\Psi(\sigma)(x)= x \cdot \sigma\inverse$
for $\sigma\in \Sigma_r$ and $x\in T^{n,r}$.  Then classical Schur-Weyl
duality is the pair of equalities
\begin{equation}
  \label{eq:classicalSWD}
  \Phi(k\GL_n(k)) = \End_{\Sigma_r}(T^{n,r})\quad \text{and} \quad
  \Psi(k\Sigma_r) = \End_{\GL_n(k)}(T^{n,r}) .
\end{equation}

It will be convenient to define a Lie idempotent to be any idempotent $e$ in
$k\Sigma_r$ such that $L^{n,r}= T^{n,r} \cdot e$. This agrees with the usual
definition (see e.g.~\cite[\S8.4]{reutenauer:free}) when $n\geq r$. Lie
idempotents exist whenever the characteristic of $k$ does not divide
$r$. Proofs of the following well-known result may be found in
\cite[\S2]{garsia:combinatorics} and \cite[\S8.4]{reutenauer:free}.

\begin{lemma}
  Assume that the characteristic of $k$ does not divide $r$.  Then the
  Dynkin-Specht-Wever idempotent $e_r= \frac 1r (1-\gamma_2) \dotsm
  (1-\gamma_r)$, where for $2\leq i\leq r$, $\gamma_i$ is the descending
  $i$-cycle $(i\ \dotsm \ 2\ 1)$ in $\Sigma_r$, is a Lie idempotent.
\end{lemma}

Suppose that $e$ is a Lie idempotent.  Then $ek\Sigma_r e$ acts on $L^{n,r}$
on the right and $\GL_n(k)$ acts on $L^{n,r}$ on the left. Thus there are
$k$-algebra homomorphisms
\begin{equation*}
  \xymatrix{k\GL_n(k) \ar[r]^-{\Phi_e} &\End_k(L^{n,r})& e k\Sigma_r
    e\ar[l]_-{\Psi_e}}
\end{equation*}
such that the images of $\Phi_e$ and $\Psi_e$ commute, so $L^{n,r}$ is a
$(k\GL_n(k), ek\Sigma_r e)$-bimodule.  Our main theorem is the following
analogue of classical Schur-Weyl duality for this bimodule. The theorem is
proved in~\S\ref{sec:3}.

\begin{theorem}\label{thm:1}
  Suppose that $k$ is a field of cardinality strictly larger than $r$ such
  that the characteristic of $k$ does not divide $r$, and that $e$ is a Lie
  idempotent in $k\Sigma_r$. Then
  \[
  \Psi_e(ek\Sigma_r e)= \End_{\GL_n(k)}(L^{n,r}).
  \]
  If in addition $T^{n,r}$ is a direct sum of absolutely irreducible
  $k\GL_n(k)$-modules, then
  \[
  \Phi_e(k\GL_n(k)) = \End_{ek\Sigma_r e}(L^{n,r}).
  \]
\end{theorem}

Recall that the Schur algebra over $k$ is the algebra 
\[
\CS(n,r) = \Phi(k\GL_n(k)) = \End_{\Sigma_r}(T^{n,r})
\]
appearing in \eqref{eq:classicalSWD}.  Suppose that the characteristic of
$k$ is larger than $r$, so $|k|>r$ as well. It is well-known that in this
case $k\Sigma_r$ is a split, semisimple $k$-algebra (see
\cite{james:irreducible}) and so $\Psi(k\Sigma_r)$ is a split, semisimple
$k$-algebra. By classical Schur-Weyl duality the centralizer of
$\Psi(k\Sigma_r)$ in $\End_k(T^{n,r})$ is $\CS(n,r)$. It is easy to see that
the centralizer of a split, semisimple subalgebra of $\End_k(T^{n,r})$ is
again split, semisimple (see \S\ref{ssec:ss1}), so $\CS(n,r)$ is a split,
semisimple $k$-algebra. Thus, $T^{n,r}$ is a direct sum of absolutely
irreducible $k\GL_n(k)$-modules and so both equalities in the theorem hold.

\begin{corollary}\label{cor:main}
  The triple $(k\GL_n(k), L^{n,r}, ek\Sigma_r e)$ satisfies Schur-Weyl
  duality whenever the characteristic of $k$ is strictly larger than $r$.
\end{corollary}

Since classical Schur-Weyl duality is known to hold whenever $|k|>r$ (see
\cite{bensondoty:schur}), it is natural to ask whether $(k\GL_n(k), L^{n,r},
ek\Sigma_r e)$ satisfies Schur-Weyl duality whenever $|k|>r$ and the
characteristic of $k$ does not divide $r$.  Based on small rank examples, it
seems likely that this is indeed the case. As a step in this direction,
in~\S\ref{sec:4} it is shown that the second equality in the theorem is
equivalent to a statement about intertwining operators between certain
transitive permutation representations arising from Young subgroups of
$\Sigma_r$.

Assume for a moment that $n\geq r$. Then $T^{n,r}$ is a faithful
$k\Sigma_r$-module and the Schur functor
$\mathfrak{f}\colon M \mapsto \epsilon M$ from left $\CS(n,r)$-modules to
left $k\Sigma_r$-modules may be defined, where $\epsilon\in \CS(n,r)$ is an
idempotent that projects $T^{n,r}$ onto its $(1^r, 0^{n-r})$-weight
space. By \cite[(6.3d)]{green:polynomial} we have that
$\mathfrak{f}(T^{n,r}) = \epsilon T^{n,r}$ is isomorphic to the left regular
$k\Sigma_r$-module ${}_{k\Sigma_r}k\Sigma_r$ and so \eqref{eq:classicalSWD}
takes the form
\begin{equation}\label{eq:7}
  \CS(n,r) = \End_{k\Sigma_r} (T^{n,r})  
\end{equation}
and
\begin{equation}\label{eq:8}
  \End_{k\Sigma_r} (\mathfrak{f}(T^{n,r}))= \End_{ k\Sigma_r}
  ({}_{k\Sigma_r}k\Sigma_r) \cong \End_{\CS(n,r)} (T^{n,r}) . 
\end{equation} 

Assume further that the characteristic of $k$ does not divide $r$, so that a
Lie idempotent $e$ exists. Because $L^{n,r} = T^{n,r}e$ is a
$GL_n(k)$-stable subspace of $T^{n,r}$, it has a natural $\CS(n,r)$-module
structure. The Lie module,
\begin{equation*}
  \Lie(r) = \mathfrak{f}(L^{n,r}),
\end{equation*}
is the $k\Sigma_r$-module obtained by applying the Schur functor to
$L^{n,r}$. Then
\[
\mathfrak{f}(L^{n,r}) = \epsilon L^{n,r} = \epsilon(T^{n,r} e) = (\epsilon
T^{n,r}) e \cong k\Sigma_r e,
\]
and so the Lie module is isomorphic to the left $k\Sigma_r$-module
$k\Sigma_r e$. Let $\overline{\CS(n,r)} = \Phi_e(k \GL_n(k))$. If
$\varphi\in \CS(n,r)$, then $\varphi$ is $\Sigma_r$-equivariant and so
$\varphi(L^{n,r}) \subseteq L^{n,r}$. Hence restriction defines an algebra
homomorphism from $\CS(n,r)$ to $\End_{ek\Sigma_r e} (L^{n,r})$ with image
equal to $\overline{\CS(n,r)}$. With this notation, Theorem~\ref{thm:1}
immediately implies the following analogue of classical Schur-Weyl duality
expressed by \eqref{eq:7} and \eqref{eq:8} above, with the $r\th$ graded
piece of the free associative algebra on $n$ letters replaced by the $r\th$
graded piece of the free Lie algebra on $n$ letters, and with the left
regular $k\Sigma_r$-module replaced by the $r\th$ Lie module.

\begin{corollary}
  Suppose that $|k|>r$, the characteristic of $k$ does not divide $r$, and
  $n\geq r$. Then
  \begin{equation*}
    \overline{\CS(n,r)} = \End_{e k\Sigma_r e} (L^{n,r})  
  \end{equation*}
  and
  \begin{equation*}
    \End_{e k\Sigma_r e} (\mathfrak{f}(L^{n,r}))= \End_{e k\Sigma_r e}
    (\Lie(r)) \cong \End_{\overline{\CS(n,r)}} (L^{n,r}) ,  
  \end{equation*} 
  where $e$ is any Lie idempotent.
\end{corollary}

Returning to the case of general $n$ and $r$, suppose the field $k$ contains
a primitive $r\th$ root of unity $\zeta$. Then the right ideal $e k\Sigma_r$
and the algebra $e k\Sigma_r e$ arise in a surprisingly different
context. To describe this connection further, fix an $r$-cycle $\gamma$ in
$\Sigma_r$ and let $\Gamma= \langle \gamma \rangle$. Let $f= (1/r)
\sum_{i=1}^r \zeta^{-i} \gamma^i$. Then $f$ is the primitive idempotent in
$k\Gamma$ corresponding to a faithful character of $\Gamma$. The right ideal
$fk\Sigma_r$ of $k\Sigma_r$ affords the induced representation
$\Ind_\Gamma^{\Sigma_r} \zeta$, and the subalgebra $f k\Sigma_r f$ is
isomorphic to the endomorphism algebra of the induced module
$\Ind_\Gamma^{\Sigma_r} \zeta$. There is a Lie idempotent $\kappa$, the
Klyachko idempotent (see \S\ref{sec:4}), such that $e\kappa =e$, $\kappa
f=f$, and $f\kappa =\kappa$. It follows that
\[
e k\Sigma_r \cong \kappa k\Sigma_r = f k\Sigma_r \quad \text{and so} \quad e
k\Sigma_r e \cong f k\Sigma_r f \cong \End_{k\Sigma_r}(
\Ind_\Gamma^{\Sigma_r} \zeta).
\]

On the other hand, suppose that $k=\BBC$ and let $M$ denote the subset of
$\BBC^n$ consisting of vectors with distinct coordinates. Then $M$ is the
complement of the union of the hyperplanes in the braid arrangement on $r$
strands. Arnold \cite{arnold:cohomology} has described the cohomology ring
$H^*(M)$. The group $\Sigma_r$ acts on $M$ by permuting the coordinates and
hence acts on the cohomology spaces $H^p(M)$. Lehrer and Solomon
\cite{lehrersolomon:symmetric} have described these representations of
$\Sigma_r$ as direct sums of representations induced from linear characters
of centralizers. A special case is the $r$-cycle $\gamma$ and its
centralizer $\Gamma$. In this case, it follows from the results in
\cite[\S5]{douglasspfeifferroehrle:cohomology} that the representation of
$\Sigma_r$ afforded by the highest non-vanishing cohomology space
$H^{r-1}(M)$ is isomorphic to the representation afforded by
$\BBC_{\operatorname{sgn}} \otimes f \BBC \Sigma_r \cong
\BBC_{\operatorname{sgn}} \otimes e \BBC \Sigma_r$,
where $\BBC_{\operatorname{sgn}}$ denotes the sign representation of
$\Sigma_r$.

\section{Generalized Schur-Weyl duality}\label{sec:3}
We will now prove Theorem~\ref{thm:1}. It turns out that our result is a
special case of a more general result, as formulated below.

Suppose that $A$ and $B$ are $k$-algebras and $V$ is an
$(A,B)$-bimodule. Then $V$ is a left $B\op$-module and there are
$k$-algebra homomorphisms
\begin{equation}\label{eq:1}
  \xymatrix{A \ar[r]^-{\Phi} &\End_k(V)& B\op \ar[l]_-{\Psi}} ,
\end{equation}
where $\Phi(a)(v)= a v$ and $\Psi(b)(v)= v b$ for $a\in A$, $v\in V$, and
$b\in B$.  Assume that the triple $(A,V,B)$ satisfies Schur-Weyl duality, so
\begin{equation*}
  \Phi(A)= \End_{B}(V)\quad \text{and}\quad \Psi(B)= \End_{A} (V).
\end{equation*}

Suppose that $e$ in $B$ is an idempotent such that $Ve\ne 0$. Clearly $Ve$
is an $(A,eBe)$-bimodule and we can ask under what conditions $(A, Ve, eBe)$
satisfies Schur-Weyl duality. In this situation, the commuting actions
induce $k$-algebra homomorphisms
\begin{equation*}
  \xymatrix{A \ar[r]^-{\Phi_e} &\End_k(Ve)& (eBe)\op \ar[l]_-{\Psi_e} }
\end{equation*}
such that 
\begin{equation}\label{eq:6}
  \Phi_e(A)\subseteq \End_{eBe}(Ve) \quad \text{and} \quad
  \Psi_e(eBe)\subseteq \End_{A} (Ve).  
\end{equation}
We wish to find conditions under which the above inclusions are equalities;
that is, we wish to prove that $(A, Ve, eBe)$ satisfies Schur-Weyl duality
under appropriate hypotheses. That the second inclusion in~\eqref{eq:6} is an
equality is an easy general fact, requiring no additional hypothesis.

\begin{lemma}\label{lem:1}
  Suppose that $(A,V,B)$ satisfies Schur-Weyl duality, $e$ is an idempotent
  in $B$ such that $Ve\ne0$, and $\Psi_e\colon eBe\to \End_k(Ve)$ is the
  $k$-algebra homomorphism induced by the right $eBe$-module structure on
  $Ve$. Then
  \[
  \Psi_e(eBe)= \End_A(Ve).
  \]
\end{lemma}

\begin{proof}
  Set $\pi= \Psi(e)$. Then $\pi(v)= ve$ for $v$ in $V$, $\pi$ is an
  idempotent in $\End_{A}(V)$, and $Ve$ is the image of $\pi$.

  Suppose $\varphi$ is in $\End_{k}(V)$. Then $\pi\varphi\pi(Ve)\subseteq
  Ve$. We denote the restriction of $\pi \varphi \pi$ to $Ve$ by $\pi
  \varphi \pi|_{Ve}$. Then $\pi \varphi \pi|_{Ve}$ is in
  $\End_{k}(Ve)$. Define
  \[
  \Pi\colon \End_{k}(V)\to \End_{k}(Ve) \quad \text{by} \quad \Pi(\varphi)=
  \pi\varphi\pi|_{Ve}.
  \]
  Clearly $Ve$ is an $A$-submodule of $V$. If $\varphi$ is $A$-linear, then
  so is $\pi \varphi \pi|_{Ve}$. Therefore, $\Pi(\End_A(V))
  \subseteq \End_A(Ve)$.  The $A$-module decomposition $V\cong Ve\oplus
  V(1-e)$ of $V$ determines a canonical decomposition
  \begin{multline*}
    \End_A(V) \cong \\
    \End_A(Ve) \oplus \Hom_A(Ve, V(1-e)) \oplus \Hom_A(V(1-e), Ve)
    \oplus \End_A(V(1-e))
  \end{multline*}
  under which the linear map $\Pi$ is identified with the projection onto
  $\End_A(Ve)$. In particular,
  \begin{equation}\label{eq:2}
    \Pi(\End_A(V))= \End_A(Ve). 
  \end{equation}

  It is straightforward to check that $\Pi \Phi=\Phi_e$ and so we can
  extend \eqref{eq:1} to a commutative diagram
  \begin{equation*}
  \begin{gathered} 
    \xymatrix@!C{A \ar[r]^-{\Phi} \ar[dr]_-{\Phi_e} &\End_k(V) \ar[d]^{\Pi}& B
      \ar[l]_-{\Psi} \ar[d]^{\Pi_e} \\
      &\End_k(Ve) & eBe \ar[l]^-{\Psi_e} }
  \end{gathered}
  \end{equation*}
  of $k$-linear maps, where $\Pi_e\colon B\to eBe$ is given by
  $\Pi_e(b)=ebe$. Thus,
  \begin{equation*}
    \Psi_e(eBe)= \Psi_e\Pi_e(B)= \Pi \Psi(B) = \Pi(\End_A(V))
    = \End_A(Ve), 
  \end{equation*}  
  where the penultimate equality follows from the assumption that $(A,V,B)$
  satisfies Schur-Weyl duality, and the final equality follows
  from~\eqref{eq:2}.
\end{proof}

\subsection*{The semisimple case}\label{ssec:ss1}
Showing that the containment $\Phi_e(A)\subseteq \End_{eBe}(Ve)$ is an
equality is not so easy. We consider first the case when $\Phi(A)$ is a
split, semisimple $k$-algebra.

Precisely, assume that $V$ is a finite dimensional $k$-vector space, and
suppose that $(A,V,B)$ satisfies Schur-Weyl duality and that $V$ is a
completely reducible $A$-module whose irreducible constituents are
absolutely irreducible. In this case we show that if $e$ is an idempotent in
$B$ with $Ve\ne0$, then $(A, Ve, eBe)$ satisfies Schur-Weyl duality and both
the $eBe$-module structure of $Ve$ and the algebra structure of
$\Psi_e(eBe)$ are determined by the $A$-module structure of $V$. The
argument is based on a version of the double centralizer theorem in the form
given below. We include a sketch of the proof in order to establish notation
and because of the lack of a suitable reference.

\begin{theorem}\label{thm:dct}
  Suppose $k$ is a field, $V$ is a finite dimensional $k$-vector
  space, and $X$ is a subalgebra of $\End_k(V)$ such that $V$ is a
  completely reducible $X$-module whose irreducible constituents are
  absolutely irreducible. Define $Y=\End_X(V)$ and suppose $\{L_1,
  \dots, L_p\}$ is a complete set of non-isomorphic, irreducible
  $X$-modules. For $1\leq i\leq p$ define $M_i= \Hom_X(L_i,V)$. Then
  the following statements hold.
  \begin{enumerate}
  \item $X$ and $Y$ are split, semisimple $k$ algebras and $\{M_1,
    \dots, M_p\}$ is a complete set of non-isomorphic, irreducible
    $Y$-modules.
  \item The natural evaluation map $\bigoplus_{i=1}^p L_i\otimes_k M_i \to
    V$ is an $(X,Y\op)$-bimodule isomorphism.
  \item $X=\End_Y(V)$ and the triple $(X, V, Y\op)$ satisfies Schur-Weyl
    duality.
  \end{enumerate}
\end{theorem}

\begin{proof}
  It follows from \cite[(3.31)]{curtisreiner:methodsI} that $X$ is
  semisimple and that each $L_i$ occurs as an irreducible constituent of
  $V$. Let $V_1$, \dots, $V_p$ be the homogeneous components of $V$ where
  $V_i\cong L_i^{m_i}$, so $V\cong \bigoplus_{i=1}^p V_i \cong
  \bigoplus_{i=1}^p L_i^{m_i}$. Set $\dim L_i= l_i$. Then $X\cong
  \bigoplus_{i=1}^p M_{l_i}(k)$ is a split, semisimple $k$-algebra.

  Now let $Y=\End_X(V)\cong \bigoplus_{i=1}^p \End_X(V_i)$ be the
  centralizer of $X$ in $\End_k(V)$. For $1\leq i\leq p$ set
  $Y_i= \End_X(V_i)$, $M_i'= \Hom_X(L_i, V_i)$, and $M_i= \Hom_X(L_i,
  V)$.
  Then $Y_i\cong M_{m_i}(k)$ is a simple $k$-algebra,
  $Y\cong \bigoplus_{i=1}^p M_{m_i}(k)$ is a semisimple $k$-algebra with the
  property that every irreducible $Y$-module is absolutely irreducible,
  $M_i'$ is an irreducible $Y_i$-module, $M_i$ is an irreducible $Y$-module
  on which the factor $Y_i$ acts non-trivially, and $M_i'\cong M_i$ as
  $Y$-modules, where $Y$ acts on $M_i'$ via the projection to $Y_i$. It is
  straightforward to check that the natural evaluation map
  $\varphi_i\colon L_i\otimes_k M_i'\to V_i$ with
  $\varphi_i(l\otimes f)= f(l)$ is an isomorphism of $(X,Y_i\op)$-bimodules,
  where $X\times Y_i\op$ acts on $L_i\otimes_k M_i'$ with
  $(x,y)\cdot l\otimes f= x\cdot l\otimes y\circ f$, and on $V_i$ by
  $(x,y) \cdot v= x(y(v))$, for all $x\in X$, $y\in Y_i$, $l\in L_i$,
  $f\in M_i'$, and $v\in V_i$. This proves the first statement in the
  theorem.

  Using the isomorphisms $\varphi_i$, it is straightforward to show that the
  natural evaluation map $\bigoplus_{i=1}^r L_i\otimes_k M_i\to V$ is an
  isomorphism of $(X, Y\op)$-bimodules. 

  It follows that as a $Y$-module, $V$ is isomorphic to $\bigoplus_{i=1}^p
  M_i^{l_i}$. Because each $M_i$ is absolutely irreducible we see that
  $\End_Y(V) \cong \bigoplus_{i=1}^p M_{l_i}(k) $. Finally, because
  $X\subseteq \End_Y(V)$, it follows that $X= \End_Y(V)$. Thus $(X, V,
  Y\op)$ satisfies Schur-Weyl duality.
\end{proof}

Now suppose $(A,V,B)$ satisfies Schur-Weyl duality, $V$ is a completely
reducible $A$-module whose irreducible constituents are absolutely
irreducible, and $e$ is an idempotent in $B$ such that $Ve\ne 0$. Set
$X=\Phi(A)$ and $Y=\Psi(B\op)$. Then $X$ and $Y$ are subalgebras of
$\End_k(V)$, $V$ is a completely reducible $X$-module whose irreducible
constituents are absolutely irreducible, $Y=\End_X(V)$, and $X=\End_Y(V)$.

\begin{theorem}\label{thm:2}
  With the assumptions and notation above, the following statements hold.
  \begin{enumerate}
  \item The subalgebra $\Psi_e(eBe)$ of $\End_k(Ve)$ is a split,
    semisimple $k$-algebra and $\{\, M_ie \mid M_ie\ne0\,\}$ is a
    complete set of non-isomorphic, irreducible right
    $\Psi_e(eBe)$-modules.
  \item $Ve$ is a completely reducible right $eBe$-module and $\{\,
    M_ie \mid M_ie\ne0\,\}$ is a complete set of non-isomorphic,
    irreducible right $eBe$-modules that occur as constituents of
    $Ve$.
  \item The triple $(A, Ve, eBe)$ satisfies Schur-Weyl duality.
  \end{enumerate}
\end{theorem}

\begin{proof}
  Set $\pi=\Psi(e)$, so $\pi$ is a non-zero idempotent in
  $Y$. By Theorem~\ref{thm:dct} and the general theory of split, semisimple
  algebras, $\pi Y \pi$ is a split, semisimple $k$-algebra and
  $\{\, \pi M_i \mid \pi M_i\ne0\,\}$ is a complete set of non-isomorphic,
  irreducible $\pi Y \pi$-modules. Now $\pi(V)=Ve$ and the image of the
  homomorphism $\Psi_e\colon (eBe)\op \to \End_k(Ve)$ coincides with the
  image of the natural homomorphism $\Psi_\pi\colon \pi Y\pi \to
  \End_k(\pi(V))$.
  Moreover, by definition $\pi M_i= M_ie$ for $1\leq i\leq p$. Therefore,
  $\Psi_e(eBe)$ is a split, semisimple $k$-algebra and
  $\{\, M_ie \mid M_ie\ne0\,\}$ is a complete set of non-isomorphic,
  irreducible right $\Psi_e(eBe)$-modules.

  The algebra $A$ acts on $\bigoplus_{i=1}^p L_i\otimes_k M_i$ through its
  left action on each $L_i$, and the algebra $B$ acts on $\bigoplus_{i=1}^p
  L_i\otimes_k M_i$ through its right action on each $M_i$. Therefore, the
  $(A,B)$-bimodule isomorphism $V\cong \bigoplus_{i=1}^p L_i\otimes_k M_i$
  induces an $(A,eBe)$-bimodule isomorphism
  \begin{equation}\label{eq:bi}
    Ve\cong \bigoplus_{M_ie\ne 0} L_i\otimes_k M_ie.
  \end{equation}
  If $M_ie\ne0$, then $M_ie$ is an absolutely irreducible
  $\Psi_e(eBe)$-module and hence an absolutely irreducible $eBe$-module. The
  bimodule isomorphism~\eqref{eq:bi} induces an isomorphism of right
  $eBe$-modules
  \[
  Ve\cong \bigoplus_{M_ie\ne0} (M_ie)^{\dim L_i},
  \]
  which proves the second statement in the theorem.

  Finally, set $X_1= \Phi_e(A)$. Then $Ve$ is a completely reducible
  $X_1$-module and so by Theorem~\ref{thm:dct}, if $Y_1= \End_{X_1}(Ve)$, then
  $X_1 = \End_{Y_1}(Ve)$. Clearly $\End_{X_1}(Ve)= \End_A(Ve)$ and by
  Lemma~\ref{lem:1}, $\End_A(Ve)= \Psi_e(eBe)$. Moreover,
  $\Phi_e(A)=X_1= \End_{Y_1}(Ve)= \End_{eBe}(Ve)$, and hence $(A, Ve, eBe)$
  satisfies Schur-Weyl duality, as claimed.
\end{proof}

We can now complete the proof of Theorem~\ref{thm:1}.

\begin{proof}[Proof of Theorem~\ref{thm:1}]
  Suppose that $k$ is a field such that the cardinality of $k$ is larger
  than $r$ and the characteristic of $k$ does not divide $r$. Then by
  classical Schur-Weyl duality, the triple
  $(k\GL_n(k), T^{n,r}, k\Sigma_r)$ satisfies Schur-Weyl duality. Moreover,
  $L^{n,r} = T^{n,r} e$. Thus, it follows from Lemma~\ref{lem:1} that
  $\Psi_e(ek\Sigma_r e)=\End_{\GL_n(k)}(L^{n,r})$.

  If in addition $T^{n,r}$ is a direct sum of absolutely irreducible
  $k\GL_n(k)$-modules, then it follows from Theorem~\ref{thm:2} that
  $(k\GL_n(k), L^{n,r}, e k\Sigma_r e)$ satisfies Schur-Weyl duality, and so
  in particular, $\Phi_e(k\GL_n(k)) = \End_{ek\Sigma_r e}(L^{n,r})$. This
  completes the proof of the theorem.
\end{proof}

\subsection*{Idempotents}
We now return to the general situation considered at the beginning of this
section where the triple $(A,V,B)$ satisfies Schur-Weyl duality and $e\in B$
is an idempotent with $Ve\ne0$. We give various conditions that are
equivalent to the assertion that the triple $(A,Ve,eBe)$ satisfies
Schur-Weyl duality.

To start, consider the $k$-algebra homomorphisms
\[
\Phi' \colon A\to \End_B(V) \quad\text{and}\quad \Phi_e' \colon
A\to \End_{eBe}(Ve)
\]
induced by $\Phi$ and $\Phi_e$, respectively. Because $(A,V,B)$ satisfies
Schur-Weyl duality, $\Phi'$ is surjective, and by Lemma~\ref{lem:1},
$(A,Ve,eBe)$ satisfies Schur-Weyl duality if and only if $\Phi_e'$ is
surjective.

Suppose $\varphi$ is in $\End_{B}(V)$. Then $\varphi(Ve)=\varphi(V)e
\subseteq Ve$ and the restriction of $\varphi$ to $Ve$ induces an
$eBe$-linear homomorphism $\bar \varphi\colon Ve\to Ve$. Define
\[
  \Theta_e\colon \End_B(V)\to \End_{eBe}(Ve)\quad \text{by} \quad
  \Theta_e(\varphi)= \bar\varphi.
\]
It follows immediately from the definitions that $\Theta_e\Phi'=
\Phi_e'$. Because $\Phi'$ is surjective, it then follows that $\Phi_e'$ is
surjective if and only if $\Theta_e$ is surjective. Clearly $\Theta_e$ is
surjective if and only if every $eBe$-linear endomorphism of $Ve$ extends to
a $B$-linear endomorphism of $V$. This proves the next lemma.

\begin{lemma}\label{lem:2}
  Suppose that $(A,V,B)$ satisfies Schur-Weyl duality and $e$ is an
  idempotent in $B$ such that $Ve\ne0$. Then the following are equivalent.
  \begin{enumerate}
  \item The triple $(A, Ve, eBe)$ satisfies Schur-Weyl duality.
  \item Every $eBe$-linear endomorphism of $Ve$ extends to a $B$-linear
    endomorphism of $V$.
  \end{enumerate}
\end{lemma}

Note that the second condition in the lemma depends only on $B$ and the
right $B$-module structure on $V$, and not on the algebra $A$. This
observation can be used to replace the idempotent $e$ by any suitably
equivalent idempotent $f$, as we now explain.  Suppose $f$ is an idempotent
in $B$ such that
\begin{equation*}\label{eq:4}
  ef=f \quad\text{and}\quad fe=e.
\end{equation*}
Then the maps $\rho_f\colon Ve\to Vf$ and $\rho_e\colon Vf\to Ve$ given by
$\rho_f(x)=xf$ and $\rho_e(x)=xe$ are mutual inverses. It is straightforward
to check that
\[
\Xi\colon \End_{eBe}(Ve)\to \End_{fBf}(Vf)\quad\text{by}\quad \Xi(\varphi)=
\rho_f \varphi \rho_e
\] 
is an algebra isomorphism, with inverse $\Xi\inverse(\psi)= \rho_e\psi
\rho_f$ for $\psi$ in $\End_{fBf}(Vf)$. It is also straightforward to check
that $\Xi\, \Theta_e= \Theta_f$. This proves the next lemma.

\begin{lemma}\label{lem:3}
  With the notation as above, $\Theta_e$ is surjective if and only if
  $\Theta_f$ is surjective.
\end{lemma}

The next theorem follows from Lemma~\ref{lem:2} and Lemma~\ref{lem:3}.

\begin{theorem}\label{thm:main1}
  Suppose that $(A,V,B)$ satisfies Schur-Weyl duality, and $e$ and $f$ are
  idempotents in $B$ such that $Ve\ne0$, $ef=f$, and $fe=e$. Then the
  following are equivalent.
  \begin{enumerate}
  \item $(A, Ve, eBe)$ satisfies Schur-Weyl duality.
  \item Every $eBe$-linear endomorphism of $Ve$ extends to a $B$-linear
    endomorphism of $V$.
  \item $(A, Vf, fBf)$ satisfies Schur-Weyl duality.
  \item Every $fBf$-linear endomorphism of $Vf$ extends to a $B$-linear
    endomorphism of $V$.
  \end{enumerate}
\end{theorem}

\section{Complements}\label{sec:4}

In this section we use the results in the previous section first to
investigate the commuting algebra of the $\GL_n(k)$-action on $L^{n,r}$ when
everything is semisimple, and second to characterize when the triple
$(k\GL_n(k), L^{n,r},e k\Sigma_r e)$ satisfies Schur-Weyl duality, in terms
of certain permutation representations of $\Sigma_r$. Throughout this
section we assume that $|k|>r$, that the characteristic of $k$ does not
divide $r$, and that $k$ contains a primitive $r\th$ root of unity $\zeta$.

If $e=\sum_{\sigma\in \Sigma_r} a_\sigma \sigma$ is any idempotent in
$k\Sigma_r$, then a result of Littlewood (see \cite[Exercise
9.16]{curtisreiner:methodsI} shows that the character of the right
$k\Sigma_r$-module $ek\Sigma_r$, evaluated at a permutation $\tau$, is the
sum $\sum_{\sigma\in \CC} a_\sigma$, where $\CC$ is the conjugacy class of
$\tau$. When $e$ is a Lie idempotent, Garsia \cite[Theorem
5.2]{garsia:combinatorics} gives a formula for the sums
$\sum_{\sigma\in \CC} a_\sigma$. This formula does not depend on the choice
of Lie idempotent. Therefore, up to isomorphism, the right ideal
$ek\Sigma_r$ does not depend on the choice of $e$ and so the following lemma
follows from Lemma~\ref{lem:2}.

\begin{lemma}\label{lem:li}
  Suppose $e$ and $e'$ are Lie idempotents. Then
  $(k\GL_n(k), L^{n,r}, e k\Sigma_r e)$ satisfies Schur-Weyl duality if and
  only if $(k\GL_n(k), L^{n,r}, e' k\Sigma_r e')$ does.
\end{lemma}

By the lemma, we may choose $e$ to be any convenient Lie idempotent. In this
section we use a Lie idempotent found by Klyachko.

Given a permutation $\sigma$, an integer $i$ is a descent of $\sigma$ if
$\sigma(i)> \sigma(i+1)$. Let $\CD(\sigma)$ denote the set of descents of
$\sigma$. By definition, the major index of $\sigma$ is
\begin{equation*}
\maj(\sigma)= \sum _{i\in \CD(\sigma)} i.
\end{equation*}
The Klyachko idempotent is the element
\begin{equation*}
\kappa= \frac 1r \sum_{\sigma \in \Sigma_r} \zeta^{\maj(\sigma)} \sigma
\end{equation*}
in $k\Sigma_r$. Klyachko \cite{klyachko:lie} has shown that $\kappa$ is a
Lie idempotent. Furthermore, if $\gamma$ is any fixed $r$-cycle in
$\Sigma_r$ and we define
\begin{equation}\label{eq:define-f}
f= \frac 1r \sum_{i=1}^r \zeta^{-i} \gamma^i
\end{equation}
as in \S\ref{sec:2}, then $\kappa f=f$ and $f\kappa =\kappa$
(see~\cite[\S8.4]{reutenauer:free}). Set
\[
H=fk\Sigma_r f.
\]
Then for any Lie idempotent $e$ we have
\begin{equation}\label{eq:9}
  e k\Sigma_r e \cong \kappa k\Sigma_r \kappa \cong H
  \cong \End_{k\Sigma_r}( \Ind_\Gamma^{\Sigma_r} \zeta)  
\end{equation}
where $\Gamma = \langle \gamma \rangle$ is the subgroup generated by
$\gamma$.

\subsection*{The semisimple case}

Now assume that the characteristic of $k$ is greater than $r$, so
$k\Sigma_r$ and $H$ are split, semisimple $k$-algebras, and consider the
commuting algebra $\End_{\GL_n(k)}( L^{n,r})$. By Corollary~\ref{cor:main},
the triple $(k \GL_n(k), L^{n,r}, \kappa k\Sigma_r \kappa)$ satisfies
Schur-Weyl duality and so by Theorem~\ref{thm:main1},
$(k \GL_n(k), T^{n,r}f, H)$ does as well. Note that right multiplication by
$\kappa$ defines a $\GL_n(k)$-equivariant isomorphism between $T^{n,r}f$ and
$L^{n,r}$ that intertwines the right actions of $H$ and
$\kappa k\Sigma_r \kappa$, and that by \eqref{eq:9},
\begin{equation*}
\End_{\GL_n(k)}( L^{n,r}) \cong H \cong \End_{k\Sigma_r}(
\Ind_\Gamma^{\Sigma_r} \zeta).
\end{equation*}
In the following, we consider the algebra $H$ instead of $\End_{\GL_n(k)}(
L^{n,r})$.

Recall that a partition is a sequence $\lambda=(\lambda_1, \lambda_2,
\dots)$ of non-negative integers such that (1) $\lambda_1\geq \lambda_2\geq
\dots$ and (2) $\lambda_i\ne0$ for only finitely many $i$. If $\lambda_i>0$,
then $\lambda_i$ is a part of $\lambda$. Define $\ell(\lambda)$ to be the
number of parts of $\lambda$. If $\sum_{i>0} \lambda_i=r$, then say
$\lambda$ is a partition of $r$ and write $\lambda \vdash r$. When
$\ell(\lambda)=a$ we generally abuse notation and write $\lambda=(\lambda_1,
\dots, \lambda_a)$ instead of $\lambda=(\lambda_1, \dots, \lambda_a, 0,
\dots)$.

For a partition $\lambda$ of $r$ with at most $n$ parts let $V^\lambda$ be
the irreducible representation of $\GL_n(k)$ with highest weight $\lambda$
and let $S^\lambda$ be the Specht module indexed by $\lambda$. For example,
if $\lambda=(r)$, then $V^\lambda$ is the natural module for $\GL_n(k)$ and
$S^\lambda$ is the trivial representation of $\Sigma_r$, and if all the
parts of $\lambda$ are equal to $1$, then $V^\lambda$ is the determinant
representation of $\GL_n(k)$ and $S^\lambda$ is the sign representation of
$\Sigma_r$.  Semisimplicity implies (see e.g.~\cite[Proposition
3.3.2]{goodmanwallach:representations}) that there is an isomorphism of
$(k\GL_n(k), k\Sigma_r)$-bimodules
\begin{equation}\label{eq:5}
  T^{n,r}\cong \bigoplus_{\substack{\lambda \vdash r\\ \ell(\lambda)\leq
      n}} V^\lambda  \otimes S^\lambda .
\end{equation}
Applying the map $\rho_f$ to~\eqref{eq:5} gives
\begin{equation*}\label{eq:5f}
  T^{n,r}f \cong \bigoplus_{\substack{\lambda \vdash r\\ \ell(\lambda)\leq n}}
  V^\lambda \otimes S^\lambda f
\end{equation*}
where $\{\, S^\lambda f \mid S^\lambda f \ne 0\,\}$ is a set of
non-isomorphic, irreducible, right $H$-modules. For a partition $\lambda$ of
$r$ with $S^\lambda f \ne0$, let $H_\lambda$ denote the minimal two-sided
ideal of $H$ with the property that $S^\lambda f H_\lambda \ne 0$.

The decomposition of $k\Sigma_r f$ into irreducible constituents is given in
\cite[Chapter 8]{reutenauer:free} and~\cite{garsia:combinatorics}. This
decomposition determines the algebra structure of $H$ as follows.

Let $\SYT$ denote the set of standard Young tableaux with $r$
boxes. For a partition $\lambda$ of $r$ let $\SYT^\lambda$ be the set
of standard Young tableaux with shape $\lambda$. Suppose $\mathsf{t}$
is a standard Young tableau. An integer $i$ is a descent of
$\mathsf{t}$ if $i+1$ occurs in a lower row of $\mathsf{t}$ than
$i$. Let $\CD(t)$ denote the set of descents of $\mathsf{t}$. The
major index of $\mathsf{t}$ is
\[
\maj(\mathsf{t})= \sum _{i\in \CD(t)} i.
\] 
Define 
\[
\SYT^\lambda_{\scriptscriptstyle \equiv 1}= \{\, \mathsf{t} \in \SYT^\lambda
\mid \maj(\mathsf{t}) \equiv 1 \bmod r\,\}.
\]
Then the multiplicity of $S^\lambda$ in $k\Sigma_r f$ is
$|\SYT^\lambda_{\scriptscriptstyle \equiv 1}|$. Thus,
\begin{enumerate}
\item[(H1)] the simple $H$-modules are parametrized by the set of partitions
  $\lambda$ of $r$ for which $|\SYT^\lambda_{\scriptscriptstyle \equiv 1}|
  \ne \emptyset$,
\item[(H2)] the dimension of $S^\lambda f$ is
  $|\SYT^\lambda_{\scriptscriptstyle \equiv 1}|$,
\item[(H3)] the dimension of $H_\lambda$ is
  $|\SYT^\lambda_{\scriptscriptstyle \equiv 1}|^2$, and
\item[(H4)] $\dim H= \sum_{\lambda\vdash r}
  |\SYT^\lambda_{\scriptscriptstyle \equiv 1}|^2$.
\end{enumerate}

Obviously $|\SYT^\lambda_{\scriptscriptstyle \equiv 1}|$ depends on only the
integer $r$, and not the field $k$, so one might hope that statements
(H1)--(H4) hold whenever the characteristic of $k$ does not divide $r$.

The partitions $\lambda$ such that $S^\lambda f\ne 0$ have been determined
by Klyachko, as follows.

\begin{theorem}\label{thm:k}
  Suppose that $\lambda$ is a partition of $r$. Then there is a standard
  Young tableau with shape $\lambda$ and major index congruent to $1$ mod
  $r$ if and only if $\lambda$ is not equal to $(1,1, \dots,1)$, $(r)$,
  $(2,2)$ (in case $r=4$), or $(2,2,2)$ (in case $r=6$).
\end{theorem}

We can use the Robinson-Schensted correspondence between $\Sigma_r$ and the
set of pairs of standard Young tableaux with the same shape to obtain
formulas for the dimensions of $H_\lambda$ and $H$ in terms of permutations
instead of tableaux. Let $P\colon \Sigma_r\to \SYT$ be the map given by the
Schensted (row) insertion algorithm. Then the Robinson-Schensted
correspondence is given by the assignment $\sigma\mapsto (P(\sigma),
P(\sigma\inverse))$.

\begin{proposition}
  Suppose that the characteristic of $k$ is greater than $r$ and that $k$
  contains a primitive $r\th$ root of unity. Then
  \[
  \dim H_\lambda= |\{\,\sigma\in \Sigma_r \mid P(\sigma)\in \SYT^\lambda,\
  \maj(\sigma) \equiv 1 \bmod r,\ \maj(\sigma\inverse) \equiv 1 \bmod r\,\}|
  \]
  and
  \[
  \dim H= |\{\,\sigma\in \Sigma_r \mid \maj(\sigma) \equiv 1 \bmod r,\
  \maj(\sigma\inverse) \equiv 1 \bmod r\,\}|.
  \]
  Thus, the dimension of $H$ is the number of permutations $\sigma$ such
  that $\sigma$ and $\sigma\inverse$ both have major index congruent to $1$
  modulo $r$.
\end{proposition}

\begin{proof}
  Clearly, the preimage of $\SYT^ \lambda_{ \scriptscriptstyle \equiv 1}
  \times \SYT^ \lambda_{ \scriptscriptstyle \equiv 1}$ in $\Sigma_r$ is
  \[
  \{\,\sigma\in \Sigma_r \mid P(\sigma)\in \SYT^\lambda,\ \maj(P(\sigma))
  \equiv 1 \bmod r,\ \maj(P(\sigma\inverse)) \equiv 1 \bmod r\,\} .
  \]
  A standard property of the row insertion algorithm is that $\CD(\sigma)=
  \CD(P(\sigma\inverse))$ and so $\maj (\sigma)= \maj
  (P(\sigma\inverse))$. Thus, the preimage of
  $\SYT^\lambda_{\scriptscriptstyle \equiv 1} \times
  \SYT^\lambda_{\scriptscriptstyle \equiv 1}$ is the set of permutations
  $\sigma$ such that $P(\sigma)$ has shape $\lambda$, $\maj(\sigma) \equiv 1
  \bmod r$, and $\maj(\sigma\inverse) \equiv 1 \bmod r$. Thus the
  proposition follows (H3) and (H4).
\end{proof}

\subsection*{The non-semisimple case}

Now we return to the general situation, assuming only that the
characteristic of $k$ does not divide $r$, so $T^{n,r}$ is not necessarily
completely reducible. Our goal is to characterize when the triple
$(k\GL_n(k), L^{n,r},e k\Sigma_r e)$ satisfies Schur-Weyl duality in terms
of certain permutation representations of $\Sigma_r$. We continue to use the
idempotent $f\in k\Gamma$ defined in \eqref{eq:define-f}.

It was shown in Theorem~\ref{thm:main1} that the triple $(k\GL_n(k),
L^{n,r}, \kappa k\Sigma_r \kappa)$ satisfies Schur-Weyl duality if and only
if the restriction homomorphism
\[
\Theta_f \colon \End_{k\Sigma_r} (T^{n,r})\to \End_{H}(T^{n,r}f)
\]
is surjective. To streamline the notation, set $\Theta=\Theta_f$.  To find
conditions under which $\Theta$ is surjective we need to consider some
standard notation and constructions.

Denote the set $\{1, 2, \dots, n\}$ simply by $[n]$. The group $\Sigma_r$
acts on the set $[n]^r$ on the right by $(a\sigma)_j= a_{\sigma(j)}$ for $a$
in $[n]^r$ and $\sigma$ in $\Sigma_r$.  Let $\{v_1, \dots, v_r\}$ be the
standard basis of $V$. For $a=(a_1, \dots, a_r)$ in $[n]^r$ define
\[
v_a=v_{a_1} \otimes \dotsm \otimes v_{a_r}
\]
in $T^{n,r}$. Then
\[
\CB=\{\, v_a\mid a\in [n]^r\}
\]
is a $k$-basis of $T^{n,r}$. Clearly $v_a\sigma= v_{a\sigma}$ for $a$ in
$[n]^r$ and $\sigma$ in $\Sigma_r$, so $\CB$ is a $\Sigma_r$-stable subset
of $T^{n,r}$.

Technically, an element $a$ in $[n]^r$ is a function $a:[r]\to [n]$. In
particular, if $\sigma$ is in $\Sigma_r$, then $a\circ \sigma\colon [r]\to
[n]$. Thus, the right action of $\Sigma_r$ on $[n]^r$ is simply the natural
right action of $\Sigma_r$ of the set of functions $[r]\to [n]$ given by
$(a,\sigma) \mapsto a\circ \sigma$. Similarly, the group $\Sigma_n$ acts
naturally on $[n]^r$ on the left. Namely, if $a=(a_1, a_2, \dots, a_r)$ is
in $[n]^r$ and $\tau$ is in $\Sigma_n$, then
\begin{equation*}
  \tau( (a_1, a_2, \dots, a_r)) = (\tau(a_1), \tau(a_2), \dots, \tau(a_r)).
\end{equation*}
Clearly the left $\Sigma_n$-action and the right $\Sigma_r$-action commute.

For $a$ in $[n]^r$, define the content of $a$ to be the $n$-tuple
\[
\ct(a)= (|a\inverse(1)|, |a\inverse(2)|, \dots, |a\inverse(n)|).
\]
Then $\ct(a) =(m_1, m_2, \dots, m_n)$, where $m_j$ is the multiplicity of
$j$ in $a$. It is easy to see that $\ct$ is an orbit map for the right
action of $\Sigma_r$ on $[n]^r$.  In other words, $\ct(a)=\ct(b)$ if and
only if there is a $\sigma$ in $\Sigma_r$ such that $b=a\sigma$. Define
$\Lambda(n,r)$ to be the image of $\ct$. Then
\[
\Lambda(n,r)= \{\, (m_1, m_2, \dots, m_n)\in \BBN^n \mid m_1+m_2+\dots+
m_n=r\,\},
\]
and so $\Lambda(n,r)$ is the set of compositions of $r$ into at most $n$
parts, with parts of length zero allowed. It is well-known and
straightforward to check that $\Lambda(n,r)$ may be identified with the set
of weights of the group of diagonal matrices in $\GL_n(k)$ acting on the
space $T^{n,r}$. Elements of $\Lambda(n,r)$ are thus referred to as
``weights.'' In the following, we fix $n$ and $r$ and set
$\Lambda = \Lambda(n,r)$.

For a weight $\alpha$ in $\Lambda$, define
\begin{equation*}
[n]^r_\alpha= \ct\inverse(\alpha)= \{\, a\in [n]^r\mid \ct(a)=\alpha\,\},
\end{equation*}
and define the weight space $T^{n,r}_\alpha$ of $T^{n,r}$ by
\begin{equation*}
T^{n,r}_\alpha= \spn \{\, v_a\in \CB \mid a\in [n]^r_\alpha\,\} =\spn \{\,
v_a\in \CB \mid \ct(a) = \alpha\,\}.
\end{equation*}
Because $\ct(a\sigma)= \ct(a)$ for all $a\in [n]^r$ and $\sigma\in
\Sigma_r$, it follows that for each weight $\alpha$, $T^{n,r}_\alpha$ is a
right $\Sigma_r$-submodule of $T^{n,r}$ and that
\begin{equation}\label{eq:de}
  T^{n,r} \cong \bigoplus_{\alpha\in \Lambda} T^{n,r}_\alpha
\end{equation}
as right $\Sigma_r$-modules.

Now consider the subspace $T^{n,r}f$ of $T^{n,r}$. Because the left action
of $\GL_n(k)$ on $T^{n,r}$ commutes with the right action of $\Sigma_r$, it
follows that $T^{n,r}f$ is a $\GL_n(k)$-stable subspace of $T^{n,r}$, and
hence that for $\alpha\in \Lambda$, the $\alpha$ weight space of $T^{n,r}f$
is equal to $T^{n,r}_\alpha f$. Then $T^{n,r}_\alpha f$ is an $H$-submodule of
$T^{n,r}f$ and
\begin{equation}\label{eq:df}
  T^{n,r}f \cong   \bigoplus_{\alpha\in \Lambda} T^{n,r}_\alpha f  
\end{equation}
as right $H$-modules.

Recall that $\Theta\colon \End_{\Sigma_r}(T^{n,r} )\to \End_{H}(T^{n,r}f)$
is given by $\Theta(\varphi) = \bar \varphi$, where $\bar \varphi\colon
T^{n,r} \to T^{n,r} f$ is the restriction of $\varphi$. The
decompositions~\eqref{eq:de} and~\eqref{eq:df} induce isomorphisms of
$k$-vector spaces
\begin{equation}\label{eq:deE}
  \End_{\Sigma_r}(T^{n,r}) \cong \bigoplus_{\alpha, \beta \in \Lambda}
  \Hom_{\Sigma_r}( T^{n,r}_\alpha, T^{n,r}_\beta)
\end{equation}
and
\begin{equation}\label{eq:deEf}
  \End_{H}(T^{n,r}f) \cong \bigoplus_{\alpha, \beta \in \Lambda}
  \Hom_{H}( T^{n,r}_\alpha f, T^{n,r}_\beta f).
\end{equation}
Suppose $\alpha$ and $\beta$ are in $\Lambda$. If $\psi$ is in
$\Hom_{\Sigma_r}( T^{n,r}_\alpha, T^{n,r}_\beta)$, then $\psi(T^{n,r}_\alpha
f)\subseteq T^{n,r}_\beta f$ and the restriction of $\psi$ to
$T^{n,r}_\alpha$ is in $\Hom_{H}( T^{n,r}_\alpha f, T^{n,r}_\beta
f)$. Define
\[
\Theta^\alpha_\beta \colon \Hom_{\Sigma_r}(T^{n,r}_\alpha, T^{n,r}_\beta
)\to \Hom_{H}(T^{n,r}_\alpha f, T^{n,r}_\beta f) \quad \text{by}\quad
\Theta^\alpha_\beta(\psi) = \bar \psi,
\]
where $\bar \psi\colon T^{n,r}_\alpha f \to T^{n,r}_\beta f$ is the
restriction of $\psi$. The maps $\Theta$ and $\Theta^\alpha_\beta$ are
compatible with the decompositions \eqref{eq:deE} and \eqref{eq:deEf} in the
sense that the diagram
\begin{equation*}
\begin{gathered} 
\xymatrix{ \End_{\Sigma_r}(T^{n,r}) \ar[r]^-{\cong} \ar[d]^-{\Theta} &
  \bigoplus_{\alpha, \beta \in \Lambda} \Hom_{\Sigma_r}( T^{n,r}_\alpha,
  T^{n,r}_\beta) \ar[d]^-{\bigoplus
    \Theta^\alpha_\beta} \\
  \End_{H}(T^{n,r}f) \ar[r]^-{\cong}& \bigoplus_{\alpha, \beta \in \Lambda}
  \Hom_{H}( T^{n,r}_\alpha f, T^{n,r}_\beta f)}
\end{gathered}
\end{equation*}
commutes. Therefore, $\Theta$ is surjective if and only if
$\Theta^\alpha_\beta$ is surjective for all $\alpha$ and $\beta$ in
$\Lambda$. The next proposition thus follows from Lemma~\ref{lem:li},
Lemma~\ref{lem:2}, and Theorem~\ref{thm:main1}.

\begin{proposition}\label{pro:wtspc}
  Suppose $e$ is a Lie idempotent. Then $(k\GL_n(k), L^{n,r}, e k\Sigma_r
  e)$ satisfies Schur-Weyl duality if and only if the maps
  \[
  \Theta^\alpha_\beta\colon \Hom_{\Sigma_r}( T^{n,r}_\alpha, T^{n,r}_\beta)
  \to \Hom_{H}( T^{n,r}_\alpha f, T^{n,r}_\beta f)
  \]
  are surjections for all $\alpha$ and $\beta$ in $\Lambda(n,r)$.
\end{proposition}

Next, suppose that $\alpha$ is a weight in $\Lambda$. Up to the left action
of $\Sigma_n$, we may assume that
$\alpha=(m_1, m_2, \dots, m_p, 0, \dots, 0)$ where $(m_1,m_2, \dots, m_p)$
is a composition of $r$ with no parts that equal zero.  Let
\[
\Sigma_\alpha \cong \Sigma_{m_1}\times \Sigma_{m_2} \times \dots \times
\Sigma_{m_p}
\]
be the corresponding Young subgroup of $\Sigma_r$. The transitive action of
$\Sigma_r$ on $[n]^r_\alpha$ induces an isomorphism of right
$k\Sigma_r$-modules $T^{n,r}_\alpha \cong k_{\alpha} \otimes_{k
  \Sigma_\alpha} k\Sigma_r$, where $k_{\alpha}$ is the trivial right
$k\Sigma_\alpha$-module, as follows.

Given $a=(a_1, a_2, \dots, a_r)$ in $[n]^r_\alpha$, replace the occurrences
of $1$ in $a$ from left to right with $1,2, \dots, m_1$, then replace the
occurrences of $2$ in $a$ from left to right by $m_1+1, m_1+2, \dots,
m_1+m_2$, and so on. Define $\sigma_a$ to be the permutation given in one
line notation by the resulting $r$-tuple. For example if $r=8$,
$\alpha=(4,2,2, 0, \dots, 0)$, and $b=(2,1,1,3,2,1,1,3)$, then in one line
notation $\sigma_b= (5, 1,2, 7,6, 3,4, 8)$; that is, $\sigma_b(1)=5$,
$\sigma_b(2)=1$, and so on. It is easy to see that the assignment $a\mapsto
\sigma_a$ defines a bijection between $[n]^r_\alpha$ and the set
$\Sigma^\alpha$ of minimal length right coset representatives of
$\Sigma_\alpha$ in $\Sigma_r$, and that the assignment $a\mapsto
\Sigma_\alpha \sigma_a$ defines a $\Sigma_r$-equivariant bijection between
$[n]^r_\alpha$ and the set of right cosets $\Sigma_\alpha\backslash
\Sigma_r$. Thus, the assignment $v_a\mapsto 1\otimes \sigma_a$ defines an
isomorphism of right $k\Sigma_r$-modules
\begin{equation*}
h_\alpha\colon T^{n,r}_\alpha\xrightarrow{\ \cong\ } k_{\alpha}
\otimes_{k\Sigma_\alpha} k\Sigma_r.
\end{equation*}
To simplify the notation, set $M^\alpha= k_{\alpha} \otimes_{k\Sigma_\alpha}
k\Sigma_r$.

Now suppose that $\alpha, \beta\in \Lambda$. The assignment $\varphi\mapsto
h_\beta \varphi h_\alpha\inverse$ defines isomorphisms of $k$-vector spaces
\[
\Hom_{\Sigma_r}(T^{n,r}_\alpha, T^{n,r}_\beta) \xrightarrow{\ \cong \ }
\Hom_{\Sigma_r}(M^\alpha, M^\beta)
\]
and
\[
\Hom_{H}(T^{n,r}_\alpha f, T^{n,r}_\beta f) \xrightarrow{\ \cong \ }
\Hom_{H}(M^\alpha f, M^\beta f),
\]
such that the diagram
\begin{equation*}
\begin{gathered} 
\xymatrix{ \Hom_{\Sigma_r}(T^{n,r}_\alpha , T^{n,r}_\beta )
  \ar[r]^-{\Theta^\alpha_\beta} \ar[d]_{\cong} & \Hom_{H}(T^{n,r}_\alpha f,
  T^{n,r}_\beta  f) \ar[d]_{\cong}\\
  \Hom_{\Sigma_r}(M^\alpha, M^\beta) \ar[r]^-{\theta^\alpha_\beta} &
  \Hom_{H}(M^\alpha f, M^\beta f) }
\end{gathered}
\end{equation*}
commutes, where the map $\theta^\alpha_\beta$ on the bottom is again given
by restriction. Obviously $\theta^\alpha_\beta$ is surjective if and only if
$\theta^\alpha_\beta$ is. Combining this observation with
Proposition~\ref{pro:wtspc} we obtain the following corollary.

\begin{corollary}
  Suppose $e$ is a Lie idempotent. Then $(k\GL_n(k), L^{n,r}, e k\Sigma_r
  e)$ satisfies Schur-Weyl duality if and only if the restriction maps
  \[
  \theta^\alpha_\beta\colon \Hom_{\Sigma_r}( M^\alpha, M^\beta) \to
  \Hom_{H}( M^\alpha f, M^\beta f)
  \]
  are surjections for all $\alpha$ and $\beta$ in $\Lambda(n,r)$.
\end{corollary}

Thus we arrive at the following problem.

\begin{problem}\label{prob:p}
  Find a combinatorially defined basis for $\Hom_H(M^\alpha f, M^\beta f)$
  and hence show that $\dim \Hom_H(M^\alpha f, M^\beta f)$ does not depend
  on the field $k$. \qed
\end{problem}

A solution to this problem should show that $(k\GL_n(k), L^{n,r}, e
k\Sigma_r e)$ satisfies Schur-Weyl duality whenever $k$ is a field of
characteristic not dividing $r$ that contains a primitive $r\th$ root of
unity.

\medskip

\noindent {\bf Acknowledgments:} This project was initiated in academic year
2012--13, when both authors were visiting Ulam Professors at the Department
of Mathematics at the University of Colorado Boulder; the authors are
grateful to the Department of Mathematics (especially Richard Green and Nat
Thiem) for their hospitality and support. We also thank Nantel Bergeron and
Dan Nakano for helpful discussions.

\bibliographystyle{amsplain}

\end{document}